\documentclass[12pt,a4paper,draft]{amsart}
\usepackage{graphicx}
\usepackage{color}
\usepackage[colorlinks=true]{hyperref}
\hypersetup{urlcolor=blue, citecolor=blue, draft=false}

\newtheorem{theorem}{Theorem}[section]
\newtheorem{lemma}{Lemma}[section]

\theoremstyle{remark}\newtheorem{example}{Example}[section]

\theoremstyle{remark}

\usepackage{amsmath, amssymb, mathrsfs}
\begin{document}
\title[Geometric properties of some generalized Mathieu power series]
      {Geometric properties of some generalized Mathieu power series inside the unit disk}

\author[S. Gerhold, {\v Z}. Tomovski, D. Bansal, A. Soni]{Stefan Gerhold, {\v Z}ivorad Tomovski, Deepak Bansal and Amit Soni}
\thanks{{\v Z}ivorad Tomovski was supported by  DAAD,
during his visit to the Department of Physics at the University of Potsdam in Germany
from 15 June 2021 to 15 September 2021 to collaborate with Ralf Metzler.}

\address{TU Wien, Vienna, Austria}
\email{sgerhold@fam.tuwien.ac.at}

\address{University of Ostrava, Faculty of Sciences, Department of Mathematics, 
30. Dubna 22701 03 Ostrava, Czech Republic}
\email{zhivorad.tomovski@osu.cz}

\address{Department of Mathematics, University College of Engineering and Technology, Bikaner 334004, Rajasthan, India }
\email{deepakbansal\_79@yahoo.com}

\address{Department of Mathematics, Govt. Engineering College, Bikaner,  334004, Rajasthan, India }
\email{aamitt1981@gmail.com}

\date{}
\begin{abstract}
We consider two parametric families of special functions: One is defined by a power series generalizing
the classical Mathieu series, and the other one is a generalized Mathieu type power series
involving factorials in its coefficients. Using criteria due to Fej\'er and Ozaki, we provide
sufficient conditions for these functions to be close-to-convex or starlike inside the unit disk,
and thus univalent.
\end{abstract}

\maketitle

\section{Introduction and Preliminaries}

Let $\mathbb{U}=\{z\in\mathbb{C}:|z|<1\}$ denote the open unit disk and $\mathscr{A}$ denote the class of all analytic functions inside the unit disk $\mathbb{U},$ normalized by the conditions $f(0)=0,$
$f'(0)=1$.
We denote by $\mathscr{S}$ the class of all functions $f \in \mathscr{A}$ which are univalent in $\mathbb{U},$ i.e.
\begin{equation*}
   \mathscr{S}=\{f \in \mathscr{A} |\; f \;\mbox{is one-to-one
   in}\;\mathbb{U} \}.
\end{equation*}
A set $\Omega\subseteq \mathbb{C}$ containing the origin is called starlike with respect to the origin if for any point $z\in\Omega$ the line segment from the origin to $z$ lies in the interior of $\Omega$. A function $f \in \mathscr{A}$ that maps the unit disk $\mathbb{U}$ onto a starlike domain is called starlike function, and the class of such functions is denoted by $\mathscr{S}^*$. Analytically, starlike functions are characterized by the condition
\begin{equation*}
    \Re \left(\frac{zf'(z)}{f(z)} \right)>0,  \quad z
    \in \mathbb{U}.
\end{equation*}
An analytic function $f\in \mathscr{A}$ is called close-to-convex if $\Re 
\left\{e^{i\theta}zf'(z)/g(z) \right\}>0, \,z \in \mathbb{U},$ for some $\theta\in\mathbb{R}$ and for some starlike function $g\in\mathscr{S}^*$. Taking $g(z)=z$, it is easy to see that every starlike function is close-to-convex. However, the converse is not true. The Noshiro-Warschawski theorem implies that close-to-convex functions
are univalent in $\mathbb{U}$, but the converse is not true in general. Thus, it is convenient to show that $f$ is close-to-convex in order to check the univalency of $f$. Geometrically an analytic function
$f$ is called close-to-convex in $\mathbb{U}$, if the complement
of $f(\mathbb{U})$ can be written as the union of non-intersecting half-lines. These classes are studied in detail in the literature (see the books of Duren \cite{duren} and Goodman \cite{goodman}).
Among many papers dealing with these geometric properties for certain special functions, we
mention~\cite{sangal} and the references therein.

The following infinite series was named after \'{E}mile Leonard Mathieu (1835-1890), who investigated it in his 1890 monograph \cite{mathieu} on elasticity of solid bodies:
\begin{equation}\label{mathieu}
S(r)=\sum\limits_{n=1}^{\infty}{\frac{2n}{(n^2+r^2)^2}},\quad r\in \mathbb{R}_+.
\end{equation}
An integral representation of the series $S(r)$ is given by (see \cite{emersleben})
\begin{equation*}
S(r)=\frac{1}{r}\int_{0}^{\infty}{\frac{t\sin(rt)}{e^t-1}dt}.
\end{equation*}
The generalized Mathieu type power series or generalized  Mathieu type power series of fractional order $\mu$ is defined by
 (see \cite{tomovski}, \cite{tomovski1}):
\begin{equation*}
F_\mu(r;z)=\sum\limits_{n=1}^{\infty}{\frac{2n}{(n^2+r^2)^{\mu+1}}}z^n,
\quad \mu, r\in \mathbb{R}_+ ,\;|z|< 1.
\end{equation*}
In 1998, Alzer et al.~\cite{alzer} obtained the following bounds for Mathieu's series \eqref{mathieu}:
\begin{equation*}
\frac{1}{r^2+\frac{1}{2\zeta(3)}}<S(r) < \frac{1}{r^2+1/6},
\end{equation*}
where $\zeta$ denotes the zeta function. 
One can refer to \cite{alzer, cerone, choi, diananda,elezovic, makai, pogany, sriva, ourbook} about the study of Mathieu's series, its generalizations and inequalities.
In the present investigation, our aim is to study geometric properties of generalized Mathieu type power series.
It is obvious that  $z\mapsto{F}_{\mu}(r;z)\notin \mathscr{A}$, so we use the following normalization:
\begin{align}
\mathbb{F}_{\mu}(r;z)&=\frac{(r^2+1)^{\mu+1}}{2}\sum\limits_{n=1}^{\infty}{\frac{2n}{(n^2+r^2)^{\mu+1}}}z^n \notag\\
&=z+\sum\limits_{n=2}^{\infty}{\frac{n(r^2+1)^{\mu+1}}{(n^2+r^2)^{\mu+1}}}z^n. \label{MT}
\end{align}
For $\mu=1$, $\mathbb{F}_{\mu}(r;z)$ becomes a Mathieu type power series (see \cite{tomovski}). Geometric properties of  the series $\mathbb{F}_{1}(r;z)$ have already been discussed in \cite{bansaljmi}.  Recently, Gerhold et al.~\cite{gerhold} considered the following generalized Mathieu type power series:
\begin{equation*}
{Q}_{\mu}(r;z)=\sum\limits_{n=1}^{\infty}{\frac{2\;n!}{((n!)^2+r^2)^{\mu+1}}}z^n,
\quad \mu, r\in \mathbb{R}^+, |z|<1.
\end{equation*}
Since ${Q}_{\mu}(r;z)\notin \mathcal{A}$, we define the normalization
\begin{align}
{\mathbb{Q}}_{\mu}(r;z)&=\frac{(r^2+1)^{\mu+1}}{2}\sum\limits_{n=1}^{\infty}{\frac{2\;n!}{((n!)^2+r^2)^{\mu+1}}}z^n \notag \\
&= z+\sum\limits_{n=2}^{\infty}{\frac{n!(r^2+1)^{\mu+1}}{((n!)^2+r^2)^{\mu+1}}}z^n. \label{eq:def Q}
\end{align}
In the present investigation our aim is to find geometric properties of $\mathbb{F}_{\mu}(r;z)$ and ${\mathbb{Q}}_{\mu}(r;z)$.
A sequence of real numbers $\{a_{n}\}_{n\geq1}$ satisfying the condition 
\begin{equation*}
2a_{n+1}\leq a_{n}+a_{n+2},\quad n\geq1,
\end{equation*}
is called convex sequence. 
It is clear that if $f(x)$ is a convex function (of a real variable) for $x\geq1$, then the sequence $a_{n}=f(n), n=1,2...$ is convex.
We need the following lemmas to prove our main results.
\begin{lemma}\label{ozaki}{\rm(Ozaki \cite{ozaki}, Corollary~7)}.
Let $f(z)=z+\sum\limits_{n=2}^{\infty}{a_nz^n}$. Suppose
\begin{equation*}
   1\ge 2a_2\ge \cdots \ge (n+1)a_{n+1}\ge \cdots \ge 0
\end{equation*}
or
\begin{equation*}
   1\le 2a_2\le \cdots\le (n+1)a_{n+1}\le \cdots \le 2.
\end{equation*}
then $f$ is 
close-to-convex with respect to the starlike function $z/(1-z).$
\end{lemma}
\begin{lemma}{\rm(Fej\'{e}r \cite{fejer}, Satz~IX)}\label{fejer}
If $\{a_n\}$ is a non-negative real sequence with $a_1=1$ and such that
 $\{na_n\}_{n\geq1}$
and $\{na_n-(n+1)a_{n+1}\}_{n\geq1}$ are non-increasing, then the
function $f(z)=z+\sum\limits_{n=2}^{\infty}{a_nz^n}$ is in $\mathscr{S}^*$.
\end{lemma}
\begin{lemma}\label{f1}
 Let $\{a_n\}_{n\geq1}$ be a non-increasing sequence of non-negative real numbers with $a_1=1$ which is convex, i.e.
 \begin{equation*}
     a_1-a_2\geq\cdots \geq a_k-a_{k+1}\geq\cdots\geq0.
\end{equation*}
Then
\begin{equation*}
   \Re \bigg( \sum\limits_{n = 1}^\infty  a_nz^{n - 1}\bigg) > \frac12,
   \quad z\in \mathbb{U}.
\end{equation*}
\end{lemma}
\begin{proof}
  According to Lemma~3.4 in~\cite{bala}, this is due to Fej{\'e}r~\cite{fejer}. As we could
  not find the result in this reference, we give a proof for the reader's convienence, which
  is a simple variation of a proof found in another paper by Fej{\'e}r (\cite{fejer1925}, Theorem~1),
  and without claiming originality. 
  Let $z=r e^{i\theta}$ with $r\in(0,1)$ and $\theta\in(0,2\pi).$ (The cases $r=0$ and
  $\theta=0$ are  both trivial.) Define $\tilde{a}_n=r^{n-1} a_n.$ As the radius of convergence
  of $\sum a_nz^n$ is at least~$1$, we have $\tilde{a}_n=o(1).$
  The sequence $\tilde{a}_n$ inherits decrease and convexity from $a_n$ (see p.~98
  in~\cite{fejer}), and it is easy to see that it is actually \emph{strictly} convex. We define
  $s_n=1/2+\sum_{k=1}^n \cos k \theta$ and
  \[
    \sigma_n = \sum_{k=0}^n s_k = \frac12\Big( \frac{\sin\big((n+1)\theta/2\big)}{\sin(\theta/2)} \Big)^2,
  \]
  which satisfies $0\leq \sigma_n=O(1).$ Using summation by parts twice, we find, for $N\geq3,$
  \begin{align*}
   \Re  \bigg(\sum_{n = 1}^{N+1}  & a_nz^{n - 1}\bigg)
   =1 + \sum_{n=1}^N \tilde{a}_{n+1} \cos n\theta
   = 1 + \sum_{n=1}^N \tilde{a}_{n+1} (s_n-s_{n-1})\\
   &=1 + \tilde{a}_{N+1}(\sigma_N-\sigma_{N-1}) -\tilde{a}_2 s_0 -\Big(
     (\tilde{a}_{N+1}-\tilde{a}_N)\sigma_{N-1}-(\tilde{a}_3-\tilde{a}_2)\sigma_0 \\
     &\qquad-\sum_{n=3}^N\sigma_{n-2}(\tilde{a}_{n+1}-2\tilde{a}_n+\tilde{a}_{n-1})\Big).
  \end{align*}
  Then, $N\to\infty$ yields
  \begin{align*}
    \Re  \bigg(\sum_{n = 1}^{\infty}   a_nz^{n - 1}\bigg) 
    &= 1 - \tfrac12  \tilde{a}_2 +\tfrac12(\tilde{a}_3-\tilde{a}_2)
    +\sum_{n=3}^\infty \sigma_{n-2}(\tilde{a}_{n+1}-2\tilde{a}_n+\tilde{a}_{n-1}) \\
    &> 1-\tilde{a}_2 + \tfrac12 \tilde{a}_3 
    = \frac12 +\frac{\tilde{a}_1-2\tilde{a}_2+\tilde{a}_3}{2} \geq \frac12. \qedhere
  \end{align*}
\end{proof}

\section{Close-to-convexity and starlikeness of $\mathbb{F}_{\mu}(r;z)$}

\begin{theorem}\label{t2}
If $\mu>0 \;\mbox{and}\; 0<r\le \sqrt{\mu},$ th\textit{}en $\mathbb{F}_{\mu}(r;z)$ is
close-to-convex with respect to the starlike function $z/(1-z).$
\end{theorem}
\begin{proof}
In view of Lemma~\ref{ozaki} it is sufficient to prove that the sequence $\{na_{n}\}$ is non-increasing. Here 
the $a_n$ are the coefficients in the series expansion of $\mathbb{F}_{\mu}(r;z)$ given by \eqref{MT}. Let
\[f(n)=na_{n}= \frac{n^{2}(r^{2}+1)^{\mu+1}}{(n^{2}+r^{2})^{\mu+1}}.\]
Now it is sufficient to show that the function  $f(x)= \frac{x^{2}(r^{2}+1)^{\mu+1}}{(x^{2}+r^{2})^{\mu+1}}$ decreases. Differentiating $f(x)$, we have
\[f'(x)=\frac{2x(r^{2}+1)^{\mu+1}}{(x^{2}+r^{2})^{\mu+2}}(r^{2}-x^2\mu), \quad x\geq1, \mu>0.\]
The denominator is positive for all $x\geq1$ and $r>0,$ therefore $f'(x)<0$ provided that $0<r\leq\sqrt\mu.$ This completes the proof.
\end{proof}
\begin{proof}[Alternative proof] Using the Bernoulli inequality we can directly prove that the sequence $\{na_{n}\}$ is non-increasing under the stated condition.
	\begin{align*}
		na_n-(n+1)a_{n+1}&=\frac{n^{2}(r^{2}+1)^{\mu+1}}{(n^{2}+r^{2})^{\mu+1}}-\frac{(n+1)^{2}(r^{2}+1)^{\mu+1}}{((n+1)^{2}+r^{2})^{\mu+1}}\notag\\
		&= \left[\frac{(r^2+1)}{(n^2+r^2)((n+1)^2+r^2)}\right]^{\mu+1}b_n,
	\end{align*}	
	where 
	\begin{align*}
		b_n&= n^2 ((n+1)^2+r^2)^{\mu+1}-(n+1)^{2}(n^{2}+r^{2})^{\mu+1}\\
		&=n^2 (n^2+2n+1+r^2)^{\mu+1}-(n^2+2n+1)(n^{2}+r^{2})^{\mu+1}\notag\\
		&=n^2(n^{2}+r^{2})^{\mu+1}\left(1+\frac{2n+1}{n^{2}+r^{2}}\right)^{\mu+1}-n^2(n^{2}+r^{2})^{\mu+1}-(2n+1)(n^{2}+r^{2})^{\mu+1}\\
		&=n^2(n^{2}+r^{2})^{\mu+1}\left[\left(1+\frac{2n+1}{n^{2}+r^{2}}\right)^{\mu+1}-1\right]-(2n+1)(n^{2}+r^{2})^{\mu+1}\\
		&\ge n^2(n^{2}+r^{2})^{\mu+1}\frac{(2n+1)(\mu+1)}{n^{2}+r^{2}}-(2n+1)(n^{2}+r^{2})^{\mu+1}  \\
		&=(2n+1)(n^{2}+r^{2})^{\mu}[n^2(\mu+1)-(n^2+r^2)]\\
		&= (2n+1)(n^{2}+r^{2})^{\mu}[n^2\mu-r^2] \ge 0,\\
	\end{align*}
	provided $\mu>0 \;\mbox{and}\; 0<r\le \sqrt{\mu}$.	
\end{proof}
\begin{theorem}\label{th3}
If $\mu>0$ and $0<r\le \sqrt\frac{{(3+5\mu)-\sqrt{17\mu^{2}+26\mu+9}}}{2},$ then $\mathbb{F}_{\mu}(r;z)$ is starlike in
$\mathbb{U}$.
\end{theorem}
\begin{proof}
We already proved in Theorem \ref{t2} that $\{na_n\}$ is non-increasing for all $0<r\le \sqrt{\mu}$.
To show that $\mathbb{F}_{\mu}(r;z)$ is starlike in $\mathbb{U}$, using Lemma
\ref{fejer}, it is sufficient to show that the sequence
$\{na_n-(n+1)a_{n+1}\}$ is also non-increasing. That is
\begin{align*}
&na_n-2(n+1)a_{n+1}+(n+2)a_{n+2}\ge 0\\
\iff& (r^2+1)^{\mu+1} \left[\frac{n^2}{(n^2+r^2)^{\mu+1}}- 2\frac{(n+1)^2}{((n+1)^2+r^2)^{\mu+1}}+\frac{(n+2)^2}{((n+2)^2+r^2)^{\mu+1}}\right]\ge 0\\
\iff&[f(n)-2f(n+1)+f(n+2)]\ge 0,
\end{align*}
where
$$f(x)=\frac{x^2}{(x^2+r^2)^{\mu+1}},\quad x\ge 1, \mu>0. $$
To show $[f(n)-2f(n+1)+f(n+2)]\ge 0$, $n=1,2,3,4,\ldots$, it is
sufficient to prove that $f(x)$ is a convex function in the real
sense or that $f''(x)\ge 0, x\ge 1$. Differentiating twice, we have
\begin{equation*}
f''(x)=\frac{2[x^4(\mu+2\mu^{2})-x^{2}(3+5\mu )r^{2}+r^4]}{(x^2+r^2)^{\mu+3}},\quad x\ge 1, \mu>0.
\end{equation*}
The denominator is positive for all $x\ge 1$ and $r>0$. Let $\phi(x)=x^4(\mu+2\mu^{2})-x^{2}(3+5\mu )r^{2}+r^4$.
Obviously $\phi'(x)=4x^{3}(2\mu^{2}+\mu)-2(3+5\mu)r^{2}x\ge 0$ for all $x\ge 1, \mu>0 $ and $0<r\le \sqrt{\frac{2(2\mu^{2}+\mu)}{5\mu+3}}$. Thus $f''(x)\ge 0$ provided $\phi(1)\ge 0$,
which in turn gives $0<r\le \sqrt\frac{{(5\mu+3)-\sqrt{17\mu^{2}+26\mu+9}}}{2}$ . This completes the proof.
\end{proof}
\begin{theorem}\label{t4}
For  $\mu>0$ and $0<r\le \sqrt{\frac{2\mu+1}{3}}$,
\begin{equation*}
   \Re\bigg( {\frac{\mathbb{F}_{\mu}(r;z)}{z}}\bigg) >\frac{1}{2},
   \quad z\in \mathbb{U}.
\end{equation*}
\end{theorem}
\begin{proof}
First we prove that
\begin{equation*}
\left\{ {{a_n}} \right\}_{n = 1}^\infty  = \left\{ \frac{n(r^2+1)^{\mu+1}}{(n^2+r^2)^{\mu+1}} \right\}_{n = 1}^\infty
\end{equation*}
is a decreasing sequence, i.e.
\begin{equation*}
   a_n-a_{n+1}\ge 0, \quad n \in \mathbb{N}.
\end{equation*}
Now
\begin{align*}
   & a_n-a_{n + 1}\geq 0\\
   \iff& (r^2+1)^{\mu+1}\left[\frac{n}{(n^2+r^2)^{\mu+1}}-\frac{n+1}{((n+1)^2+r^2)^{\mu+1}}\right]\geq 0\\
   \iff& (r^2+1)^{\mu+1}\left[f(n)-f(n+1)\right]\geq 0,
\end{align*}
where\begin{equation}\label{f2}
f(x)=\frac{x}{(x^2+r^2)^{\mu+1}},\quad x\ge 1, \mu>0.
\end{equation}
To show $f(n)-f(n+1)\ge 0$, $n=1,2,3,\ldots,$ it is sufficient to prove that $f(x)$ is a decreasing function in the real sense or
that $f'(x)<0$, $x\ge 1$. We have
\begin{equation*}
f'(x)=\frac{r^2-(1+2\mu)x^{2}}{(x^2+r^2)^{\mu+2}}\le 0,
\quad x\ge 1, \mu>0\;{\rm and}\; 0<r\le \sqrt{1+2\mu}.
\end{equation*}
Next we prove that $\left\{{{a_n}} \right\}_{n = 1}^\infty$ is a convex decreasing sequence. For this we show
\begin{equation*}
   a_{n+2}-a_{n+1} \ge a_{n+1}-a_n, \quad n \in \mathbb{N}.
\end{equation*}
Now
\begin{align*}
   & a_n-2a_{n + 1}+a_{n + 2}\geq 0\\
   \iff & (r^2+1)^{\mu+1}\left[\frac{n}{(n^2+r^2)^{\mu+1}}-2\frac{n+1}{((n+1)^2+r^2)^{\mu+1}}+\frac{n+2}{((n+2)^2+r^2)^{\mu+1}}    \right]\geq 0\\
   \iff& (r^2+1)^{\mu+1} \left[f(n)-2f(n+1)+f(n+2)\right]\geq 0,
\end{align*}
where $f(x)$ is given by \eqref{f2}. To show $\left[f(n)+f(n+2)-2f(n+1)\right]\geq0$ ,
$n=1,2,3,4,\ldots$, it suffices to prove that $f(x)$ is a convex function in the real sense or that $f''(x)\geq0$, $x\geq1$. It can be easily verified that 
\begin{equation*}
   f''(x)=\frac{2x[(2\mu^{2}+3\mu+1)x^{2}-3(\mu+1)r^{2}]}{(x^{2}+r^{2})^{\mu+3}}\geq 0   
\end{equation*}
for all
\begin{equation*}
	x\geq 1, \mu>0 \;{\rm and}\; 0<r\le \sqrt{\frac{2\mu+1}{3}}.
\end{equation*}
Thus $\left\{{{a_n}} \right\}_{n = 1}^\infty$ is a convex decreasing sequence.
Now applying Lemma \ref{f1} to $\left\{{{a_n}} \right\}_{n = 1}^\infty$, we have
\begin{equation*}\label{sfactor}
   \Re \bigg(\sum\limits_{n = 1}^\infty  {{a_n}{z^{n - 1}}}
    \bigg) > \frac12, \quad z\in \mathbb{U},
\end{equation*}
which is equivalent to
\begin{equation*}\label{sub}
   \Re\bigg(\frac{\mathbb{F}_{\mu}(r;z)}{z}\bigg) >\frac12, \quad z\in \mathbb{U}.
\end{equation*}
\end{proof}
\begin{theorem}\label{t5}
For $\mu>0 $ and $0<r\le \sqrt\frac{{(5\mu+3)-\sqrt{17\mu^{2}+26\mu+9}}}{2}$,
\begin{equation*}
   \Re\big( {\mathbb{F}'_{\mu}(r;z)}\big)
   >\frac{1}{2}, \quad z\in \mathbb{U}.
\end{equation*}
\end{theorem}
\begin{proof}
From \eqref{MT},
\begin{equation*}
   \mathbb{F}'_{\mu}(r;z)= 1+ \sum\limits_{n=2}^{\infty}{\frac{n^2(r^2+1)^{\mu+1}}{(n^2+r^2)^{\mu+1}}z^{n-1}}.
\end{equation*}
We have shown in Theorems~\ref{t2} and~\ref{th3} that the sequence
\begin{equation*}
   \frac{n^2(r^2+1)^{\mu+1}}{(n^2+r^2)^{\mu+1}}
\end{equation*}
is convex and non-increasing, and so the result follows from Lemma~\ref{f1}.
\end{proof}

\section{Close-to-convexity and starlikeness of $\mathbb{Q}_\mu(r;z)$}

\begin{theorem}\label{t21}
If $\mu>0$ and $0<r\le \sqrt{\mu},$ then $\mathbb{Q}_{\mu}(r;z)$ is close-to-convex
w.r.t.\ the starlike function $z/(1-z).$
\end{theorem}
\begin{proof}
  Again, we apply Lemma~\ref{ozaki}.
	Let $$C_n={\frac{n!(r^2+1)^{\mu+1}}{((n!)^2+r^2)^{\mu+1}}}.$$
	Using the Bernoulli inequality we can directly prove that the sequence $\{nC_n\}$ is non-increasing  under the stated condition. We have
\begin{align*}
nC_n-(n+1)C_{n+1}&=\frac{n\; n!(r^2+1)^{\mu+1}}{{[(n!)^2+r^2]^{\mu+1}}}-\frac{(n+1)\; (n+1)!(r^2+1)^{\mu+1}}{{[((n+1)!)^2+r^2]^{\mu+1}}}\notag \\
&=\left[\frac{(r^2+1)}{((n!)^2+r^2)(((n+1)!)^2+r^2)}\right]^{\mu+1}B_n, \notag \\
\end{align*}	
where
\begin{align*}
	B_n&=n\; n![((n+1)!)^2+r^2]^{\mu+1}-(n+1)(n+1)!((n!)^2+r^2)^{\mu+1} \\
&=n\;n![((n^2+2n+1)(n)!)^2+r^2]^{\mu+1}-(n^2+2n+1)\;n!\;((n!)^2+r^2)^{\mu+1} \notag \\
&=n\;n![(n^2+2n)((n)!)^2+((n)!)^2+r^2]^{\mu+1}-(n^2+2n+1)\;n!\;((n!)^2+r^2)^{\mu+1} \notag \\
&=n\;n![(n!)^2+r^2]^{\mu+1}\left[1+\frac{(n^2+2n)((n)!)^2}{(n!)^2+r^2}\right]^{\mu+1}-(n^2+2n+1)\;n!\;((n!)^2+r^2)^{\mu+1} \notag \\
&\ge n\;n![(n!)^2+r^2]^{\mu+1}\left[1+\frac{(\mu+1)(n^2+2n)((n)!)^2}{(n!)^2+r^2}\right]-(n^2+2n+1)\;n!\;((n!)^2+r^2)^{\mu+1} \notag \\
&=n\;n![(n!)^2+r^2]^{\mu+1}+n(n^2+2n)(\mu+1)((n)!)^3[(n!)^2+r^2]^{\mu}-(n^2+2n+1)\;n!\;((n!)^2+r^2)^{\mu+1} \notag \\
&\ge\;n![(n!)^2+r^2]^{\mu+1}+(n^2+2n)(\mu+1)((n)!)^3[(n!)^2+r^2]^{\mu}-(n^2+2n+1)\;n!\;((n!)^2+r^2)^{\mu+1} \notag \\
&=(n^2+2n)(\mu+1)((n)!)^3[(n!)^2+r^2]^{\mu}-(n^2+2n)\;n!\;((n!)^2+r^2)^{\mu+1} \notag \\
&=(n)!(n^2+2n)((n!)^2+r^2)^{\mu}(\mu (n!)^2-r^2)\ge 0,\notag
\end{align*}
provided that $\mu-r^2\ge0 $ or $0<r\le\sqrt{\mu}$.
\end{proof}

\begin{theorem}\label{thm:Q star}
If $\mu\geq2$ and $0<r\le \sqrt{\mu},$ then $\mathbb{Q}_{\mu}(r;z)$ is starlike in~$\mathbb U$.
\end{theorem}
\begin{proof}
   We have shown in the preceding proof that $\{nC_n\}_{n\geq1}$ is non-increasing.
   Now we prove that it is convex, provided that $\mu\geq2,$ in order to apply
   Lemma~\ref{fejer}.
       We have $nC_n=g(n),$ where
   \[
      g(x)=\frac{x \Gamma(x+1)(1+r^2)^{\mu+1}}{\big(\Gamma(x+1)^2+r^2\big)^{\mu+1}}, \quad x\geq0.
   \]
   Define
   \[
     h(x) = g(x)-g(x+1).
   \]
   Our goal is to show that $\{h(n)\}_{n\geq1}$ is non-increasing. We have
   \begin{align*}
     h(1)-h(2) &= 1-8\Big(\frac{r^2+1}{r^2+4}\Big)^{\mu+1}+18 \Big(\frac{r^2+1}{r^2+36}\Big)^{\mu+1}\\
     &>1-8\Big(\frac{r^2+1}{r^2+4}\Big)^{\mu+1} \\
     &\geq 1-8\Big(\frac{\mu+1}{\mu+4}\Big)^{\mu+1} \geq 0,
   \end{align*}
   where the last estimate follows from $\mu\geq2$ and the fact that $\Big(\frac{\mu+1}{\mu+4}\Big)^{\mu+1}$
   decreases w.r.t.\ $\mu$. (This is the only step where $\mu\geq2$ is used; the rest of
   the proof works for $\mu>0$.) Very similarly, it can be shown that
   \[
     h(2)\geq h(3)\geq h(4).
   \]
   To complete the proof, we show that $g(x)$ is a convex function for $x\geq4.$
   We have
  \[
    g''(x)=\Gamma(x+1)\big(\Gamma(x+1)^2+r^2\big)^{-\mu-3}
    (1+r^2)^{\mu+1}A(x),
  \]
  where
  \begin{multline}\label{eq:def A}
    A(x)=2(r^2+\Gamma(x+1)^2)\big[r^2-(2\mu+1)\Gamma(x+1)^2\big]\psi(x+1) \\
    +x\big[r^4-2r^2(4\mu+3)\Gamma(x+1)^2+(2\mu+1)^2\Gamma(x+1)^4\big] \psi(x+1)^2 \\
    +x\big(r^2+\Gamma(x+1)^2\big)\big(r^2-(2\mu+1)\Gamma(x+1)^2\big)\psi'(x+1),
  \end{multline}
  and $\psi=\Gamma'/\Gamma$ denotes the digamma function.
  We will apply the following estimates:
  \begin{equation}\label{eq:r mu c}
    -2r^2(4\mu+3)c+(2\mu+1)^2c^2
     \geq 0,\quad \mu>0,0\leq r\leq \sqrt{\mu}, c\geq2,
  \end{equation}
  \begin{equation}\label{eq:psi upper}
    \psi(x) < \log x -\frac{1}{2x}, \quad x>1,
  \end{equation}
  \begin{equation}\label{eq:psi lower}
    \psi(x) > \log x -\frac{1}{x}, \quad x>1,
  \end{equation}
  \begin{equation}\label{eq:trigamma}
    \psi'(x)< \frac1x + \frac{1}{x^2}, \quad x>0,
   \end{equation}
  \begin{equation}\label{eq:sqrt}
     \log(x+1)-\frac{1}{2(x+1)} \leq \sqrt{x},\quad x\geq1,
  \end{equation}
  \begin{equation}\label{eq:19/10}
    \Big(\log(x+1)-\frac{1}{x+1}\Big)^2 \geq \frac{19}{10},\quad x\geq4,
  \end{equation}
  \begin{multline}\label{eq:total}
     2(r^2+c)(r^2-(2\mu+1)c)\sqrt{x}
      +x\big(r^4 -2r^2(4\mu+3)c+(2\mu+1)^2c^2\big) \frac{19}{10} \\
      +x(r^2+c)(r^2-(2\mu+1)c)\Big(\frac{1}{x+1}+\frac{1}{(x+1)^2}\Big)\geq0, \\
     \quad x\geq 4,c\geq \Gamma(5)^2,\mu>0,0\leq r\leq \sqrt{\mu}.
  \end{multline}
  Note that~\eqref{eq:r mu c} and~\eqref{eq:total} are polynomial inequalities
  with polynomial constraints, which can be proven by cylindrical algebraic decomposition (CAD),
  using a computer algebra system. The inequalities~\eqref{eq:sqrt} and~\eqref{eq:19/10}
  are very easy to show. The estimates~\eqref{eq:psi upper} and~\eqref{eq:psi lower}
  are found on p.~288 of~\cite{mitrinovic}, and~\eqref{eq:trigamma} is an easy
  consequence of~(26) in~\cite{merkle}.
  
  For the rest of the proof, we may assume $x\geq4.$
  The factors in front of $\psi(x+1)$ and $\psi'(x+1)$ in~\eqref{eq:def A} are clearly negative, whereas
  the factor in front of $\psi(x+1)^2$ is non-negative, by~\eqref{eq:r mu c}.
  Thus, \eqref{eq:psi upper}--\eqref{eq:trigamma} imply
  \begin{multline*}
    A(x)\geq 2(r^2+\Gamma(x+1)^2)(r^2-(2\mu+1)\Gamma(x+1)^2)\Big(\log(x+1)-\frac{1}{2(x+1)}\Big) \\
    +x(r^4-2r^2(4\mu+3)\Gamma(x+1)^2+(2\mu+1)^2\Gamma(x+1)^4) \Big(\log(x+1)-\frac{1}{x+1}\Big)^2 \\
    +x(r^2+\Gamma(x+1)^2)(r^2-(2\mu+1)\Gamma(x+1)^2)\Big(\frac{1}{x+1}+\frac{1}{(x+1)^2}\Big).
  \end{multline*}
  Using~\eqref{eq:sqrt} and~\eqref{eq:19/10}, we thus obtain
  \begin{multline*}
    A(x)\geq 2(r^2+\Gamma(x+1)^2)(r^2-(2\mu+1)\Gamma(x+1)^2)\sqrt{x} \\
    +x(r^4-2r^2(4\mu+3)\Gamma(x+1)^2+(2\mu+1)^2\Gamma(x+1)^4) \frac{19}{10} \\
    +x(r^2+\Gamma(x+1)^2)(r^2-(2\mu+1)\Gamma(x+1)^2)\Big(\frac{1}{x+1}+\frac{1}{(x+1)^2}\Big).
  \end{multline*}
  This is non-negative by~\eqref{eq:total}. The convexity of~$g(x)$ for $x\geq4$ is established,
  which completes the proof.
\end{proof}

\begin{theorem}\label{t22}
	For  $\mu>0$ and $0<r\le \sqrt{\mu}$,
	\begin{equation*}
		\Re\bigg({\frac{\mathbb{Q}_{\mu}(r;z)}{z}}\bigg)>\frac{1}{2},
		\quad z\in \mathbb{U}.
	\end{equation*}
\end{theorem}
\begin{proof}
We use Lemma~\ref{f1}.
To show that
	\begin{equation*}
		\left\{ {{C_n}} \right\}_{n = 1}^\infty  = \left\{ \frac{n!(r^2+1)^{\mu+1}}{((n!)^2+r^2)^{\mu+1}} \right\}_{n = 1}^\infty
	\end{equation*}
	is a decreasing sequence, we define 
	\begin{equation*}
		\tilde{g}(x)=\frac{\Gamma(x+1)}{((\Gamma(x+1))^2+r^2)^{\mu+1}},
		\quad x\ge 1, \mu> 0.
	\end{equation*}
We have
\begin{equation*}
\tilde{g}'(x)= \frac{\Gamma'(x+1)\left[r^2-(2\mu+1)(\Gamma(x+1))^2\right]}{\left[(\Gamma(x+1))^2+r^2\right]^{\mu+2}}\le 0
\end{equation*}
when $x\ge 1$, $\mu>0$ and $0<r\le \sqrt{2\mu+1}$. Next we prove that $\{C_n\}$ is a convex sequence.
Define
\begin{equation*}
	\tilde{h}(x)=\tilde{g}(x)-\tilde{g}(x+1).
\end{equation*}
We first show that $\tilde{h}(1)\geq \tilde{h}(2)\geq \tilde{h}(3).$
We have
\begin{align*}
  \tilde{h}(1)-\tilde{h}(2) > \frac{1}{(r^2+1)^{\mu+1}} - \frac{4}{(r^2+4)^{\mu+1}} \geq0,
\end{align*}
because the map $x\mapsto x/(r^2+x)^{\mu+1}$ decreases for $x\geq r^2/\mu.$
For the same reason,
\begin{align*}
  \tilde{h}(2)-\tilde{h}(3)
  &> \frac12 \Big( \frac{4}{(r^2+4)^{\mu+1}} -  \frac{24}{(r^2+36)^{\mu+1}} \Big) \\
  &> \frac12 \Big( \frac{4}{(r^2+4)^{\mu+1}} -  \frac{36}{(r^2+36)^{\mu+1}} \Big) >0.
\end{align*}
We complete the proof by showing that the function $\tilde{g}(x)$ is
convex for $x\geq3,$ similarly as in the proof of Theorem~\ref{thm:Q star}. We have
\[
  \tilde{g}''(x) = \Gamma(x+1)(r^2 + \Gamma(x+1)^2)^{-\mu-3} \tilde{A}(x),
\]
where
\begin{multline*}
  \tilde{A}(x):=\Big(r^4-2r^2(4\mu+3)\Gamma(x+1)^2 + (2\mu+1)^2 \Gamma(x+1)^4 \Big)
  \psi(x+1)^2\\
  +\big(r^2+\Gamma(x+1)^2\big)\big(r^2 -(2\mu+1)\Gamma(x+1)^2\big) \psi'(x+1).
\end{multline*}
We state the following inequalities:
\begin{equation}\label{eq:r mu ineq}
-2r^2(4\mu+3)+(2\mu+1)^2 c\geq 0,\quad \mu>0,0\leq r\leq \sqrt{\mu}, c\geq2.
\end{equation}
\begin{equation}\label{eq:log ineq}
  \Big(\log(x+1)-\frac{1}{x+1}\Big)^2\geq1,\quad x\geq3.
\end{equation}
\begin{equation}\label{eq:frac ineq}
  \frac{1}{x+1}+\frac{1}{(x+1)^2}\leq \frac12, \quad x\geq3.
\end{equation}
\begin{equation}\label{eq:c mu ineq}
  (2\mu+1)^2-\frac{2r^2(4\mu+3)}{c}
  -\frac{2\mu+1}{2}\Big(1+\frac{r^2}{c}\Big)>0,\quad
  \mu>0,0\leq r\leq \sqrt{\mu}, c\geq5.
\end{equation}
It is very easy to show~\eqref{eq:log ineq}. The estimate~\eqref{eq:frac ineq} is obvious,
and~\eqref{eq:r mu ineq} and~\eqref{eq:c mu ineq} can be proven by computer algebra (see above).
{}From now on we assume $x\geq3.$
Since $\psi^2(x+1)$ and $\psi'(x+1)$ are non-negative, we have
\begin{multline*}
  \tilde{A}(x)\geq \big[{-2r^2(4\mu+3)\Gamma(x+1)^2} + (2\mu+1)^2 \Gamma(x+1)^4 \big]
  \psi(x+1)^2\\
  -\big(r^2+\Gamma(x+1)^2\big)(2\mu+1)\Gamma(x+1)^2 \psi'(x+1).
\end{multline*}
By~\eqref{eq:r mu ineq}, the term $[\dots]$ is $\geq0$. We then use~\eqref{eq:psi lower}
and~\eqref{eq:trigamma} and obtain
\begin{multline*}
  \tilde{A}(x)\geq \big[{-2r^2(4\mu+3)\Gamma(x+1)^2} + (2\mu+1)^2 \Gamma(x+1)^4 \big]
  \Big(\log (x+1)-\frac{1}{x+1}\Big)^2\\
  -\big(r^2+\Gamma(x+1)^2\big)(2\mu+1)\Gamma(x+1)^2 \Big(\frac{1}{x+1}+\frac{1}{(x+1)^2} \Big).
\end{multline*}
By~\eqref{eq:log ineq} and~\eqref{eq:frac ineq}, we further get
\begin{align*}
  \tilde{A}(x)&\geq \big[-2r^2(4\mu+3)\Gamma(x+1)^2 + (2\mu+1)^2 \Gamma(x+1)^4 \big]\\
  &\qquad \qquad -\frac12 \big(r^2+\Gamma(x+1)^2\big)(2\mu+1)\Gamma(x+1)^2 \\
  &=\Gamma(x+1)^4\Big(\Big[(2\mu+1)^2-\frac{2r^2(4\mu+3)}{\Gamma(x+1)^2}\Big]
  - \frac{2\mu+1}{2}\Big[1+\frac{r^2}{\Gamma(x+1)^2}\Big]\Big).
\end{align*}
This is positive by~\eqref{eq:c mu ineq}, for $x\geq3.$
\end{proof}

\begin{theorem}
For $\mu\geq2 $ and $0<r\le \sqrt{\mu}$,
\begin{equation*}
   \Re\big( {\mathbb{Q}'_{\mu}(r;z)}\big)
   >\frac{1}{2},\quad z\in \mathbb{U}.
\end{equation*}
\end{theorem}
\begin{proof}
From \eqref{eq:def Q},
\begin{equation*}
   \mathbb{Q}'_{\mu}(r;z)= 1+ \sum\limits_{n=2}^{\infty}{\frac{n\, n!(r^2+1)^{\mu+1}}{((n!)^2+r^2)^{\mu+1}}z^{n-1}}.
\end{equation*}
We have shown in Theorems~\ref{t21} and~\ref{thm:Q star} that the sequence
\begin{equation*}
   \frac{n\, n!(r^2+1)^{\mu+1}}{((n!)^2+r^2)^{\mu+1}}
\end{equation*}
is convex and non-increasing, and so the result follows from Lemma~\ref{f1}.
\end{proof}

\section{Two further examples}

We conclude the paper with two more examples of starlike Mathieu-type power series.
Here, we apply the criterion that $z+\sum_{n=2}^\infty a_n z^n$ is starlike
if $\sum_{n=2}^\infty n |a_n|<1.$ See Goodman~\cite{goodman57}
and the references given there.

\begin{example}
  Define
  \[
     \hat{S}(z) = z+\sum_{n=2}^\infty \frac{8n}{(n^2+1)^3}z^n, \quad |z|<1.
  \]
  Using an inequality due to Diananda~\cite{diananda} and the above criterion,
  it can be shown that
  \[
     \int_0^z \frac{\hat{S}(t)}{t} dt = z
     +4 \sum_{n=2}^\infty \frac{2z^n}{(n^2+1)^3}
  \]
  is starlike. Indeed, by Theorem~1 in~\cite{diananda},
  \[
    4 \sum_{n=2}^\infty \frac{2n}{(n^2+1)^3}
    =4 \Big( \sum_{n=1}^\infty \frac{2n}{(n^2+1)^3}-\frac14\Big)
    <4\Big(\frac12-\frac14\Big)=1.
  \]
\end{example}

\begin{example}
   The power series
   \[
     z + 4 \sum_{n=2}^\infty \frac{(2n-1)!!}{[(2n+1)!!+1]^2}z^n, \quad |z|<1,
   \]
   is starlike. Recall that the double factorial is defined by
   $(2n+1)!!=1\times 3\times \dots \times (2n+1).$
   Starlikeness follows from the estimate
   \begin{align*}
     \sum_{n=1}^\infty \frac{4n(2n-1)!!}{[(2n+1)!!+1]^2}
     &<2\sum_{n=1}^\infty\frac{(2n+1)(2n-1)!!-(2n-1)!!}{[(2n+1)!!+1][(2n-1)!!+1]}\\
     &= 2 \sum_{n=1}^\infty \Big[ \frac{1}{(2n-1)!!+1}
       -\frac{1}{(2n+1)!!+1} \Big] \\
&= 2\cdot \frac{1}{(2\cdot 1-1)!!+1} = 1,
   \end{align*}
   i.e.
   \[
      \sum_{n=2}^\infty \frac{4n(2n-1)!!}{[(2n+1)!!+1]^2}
      <1-\frac14<1.
   \]
\end{example}


\end{document}